\documentclass[12pt]{article}
\usepackage{amsthm}

\usepackage[latin1]{inputenc}
\usepackage{subfigure}
\usepackage{color}
\usepackage{amsmath}
\usepackage{amsthm}
\usepackage{amstext}
\usepackage{amssymb}
\usepackage{amsfonts}
\usepackage{graphicx}
\usepackage{mathrsfs}
\usepackage{bbm}
\usepackage[all]{xy}
\usepackage{mathabx}
\usepackage{tikz}
\usepackage{mathtools}
\usepackage{tabularx}
\usepackage{array}
\usepackage{enumitem}
\usepackage{listings}
\theoremstyle{plain}

\newtheorem{theorem}{Theorem}[section]
\newtheorem{conjecture}[theorem]{Conjecture}

\newtheorem{lemma}[theorem]{Lemma}

\theoremstyle{definition}

\newtheorem{question}[theorem]{Question}

\DeclareMathAlphabet{\mathpzc}{OT1}{pzc}{m}{it}

\usepackage{amssymb,amsmath,tabularx,graphicx}
\usepackage{tikz}
\usetikzlibrary[calc,intersections,through,backgrounds,arrows,decorations.pathmorphing]

\begin{document}
	
	\title{On almost $k$-covers of hypercubes}
	\author{Alexander Clifton \thanks{Department of Mathematics, Emory University, Atlanta, USA. Email: alexander.james.clifton@emory.edu. Research supported in part by a George W. Woodruff Fellowship.} \and Hao Huang \thanks{Department of Mathematics, Emory University, Atlanta, USA. Email: hao.huang@emory.edu. Research supported in part by the Collaboration Grants from the Simons Foundation.}}

\maketitle
\abstract{In this paper, we consider the following problem: what is the minimum number of affine hyperplanes in $\mathbb{R}^n$, such that all the vertices of $\{0, 1\}^n \setminus \{\vec{0}\}$ are covered at least $k$ times, and $\vec{0}$ is uncovered? The $k=1$ case is the well-known Alon-F\"uredi theorem which says a minimum of $n$ affine hyperplanes is required, proved by the Combinatorial Nullstellensatz.

We develop an analogue of the Lubell-Yamamoto-Meshalkin inequality for subset sums, and completely solve the fractional version of this problem, which also provides an asymptotic answer to the integral version for fixed $n$ and $k \rightarrow \infty$. We also use a Punctured Combinatorial Nullstellensatz developed by Ball and Serra, to show that a minimum of $n+3$ affine hyperplanes is needed for $k=3$, and pose a conjecture for arbitrary $k$ and large $n$.

}
\section{Introduction}
Alon's Combinatorial Nullstellensatz \cite{nullstellensatz} is one of the most powerful algebraic tools in modern combinatorics. Alon and F\"uredi \cite{alon-furedi} used this method to prove the following elegant result: any set of affine hyperplanes that covers all the vertices of the $n$-cube $Q^n=\{0, 1\}^n$ but one contains at least $n$ affine hyperplanes. There are many generalizations and analogues of this theorem: for rectangular boxes \cite{alon-furedi}, Desarguesian affine and projective planes \cite{bs, jamison}, quadratic surfaces and Hermitian varieties in $PG(n, q)$ \cite{bbs}. The common theme of these results are: in many point-line (point-surface) geometries, to cover all the points except one, more lines are needed than to cover all points.

In this paper, we consider the following generalization of the Alon-F\"uredi theorem. Let $f(n, k)$ be the minimum number of affine hyperplanes needed to cover every vertex of $Q^n$ at least $k$ times (except for $\vec{0}= (0, \cdots, 0)$ which is not covered at all). For convenience, from now on we call such a cover an {\it almost $k$-cover} of the $n$-cube. The Alon-F\"uredi theorem gives $f(n, 1)=n$ since the affine hyperplanes $x_i=1$, for $i=1, \cdots, n$ covers $Q^n\setminus\{\vec{0}\}$. Their result also leads to  $f(n, 2)=n+1$. The lower bound follows from observing that when removing one hyperplane from an almost $2$-cover, the remaining hyperplanes form an almost $1$-cover. On the other hand, the $n$ affine hyperplanes $x_i=1$, together with $x_1 + \cdots + x_n=1$ cover every vertex of $Q^n\setminus\{\vec{0}\}$ at least twice.

These observations immediately lead to a lower bound $f(n, k) \ge n+k-1$ by removing $k-1$ affine hyperplanes, and an upper bound $f(n, k) \le n+ \binom{k}{2}$ by considering the following almost $k$-cover: $x_i=1$ for $i=1, \cdots, n$, together with $k-t$ copies of $\sum_{i=1}^n x_i=t$, for $t=1, \cdots, k-1$. In this construction, every binary vector with $t$ $1$-coordinates is covered $t$ times by $\{x_i=1\}$, and $k-t$ times by $x_1 + \cdots + x_n=t$. The total number of hyperplanes is $n+\sum_{t=1}^{k-1} (k-t) = n + \binom{k}{2}$.	
	
	
Note that for $k=3$, the inequalities above give $n+2 \le f(n, 3) \le n+3$. 
We used a punctured version of the Combinatorial Nullstellensatz, developed by Ball and Serra \cite{ball-serra} to show that the upper bound is tight in this case. We also improve the lower bound for $k \ge 4$.

\begin{theorem}\label{thm_3}
For $n \ge 2$, $$f(n, 3)=n+3.$$
For $k \ge 4$ and $n \ge 3$,
$$n+k+1 \le f(n, k) \le n + \binom{k}{2}.$$
\end{theorem}	

Our second result shows that for fixed $n$ and the multiplicity $k \rightarrow \infty$, the aforementioned upper bound $f(n, k) \le n+ \binom{k}{2}$ is indeed far from being tight. Indeed $f(n , k) \sim c_n k$ when $k \rightarrow \infty$. Note that $f(n, k)$ is the optimum of an integer program. We consider the following linear relaxation of it: we would like to assign to every affine hyperplane $H$ in $\mathbb{R}^n$ a non-negative weight $w(H)$, with the constraints
$$\sum_{\vec{v} \in H} w(H) \ge k,~~~\textrm{for~every~}\vec{v} \in Q^n \setminus \{\vec{0}\},$$
and 
$$\sum_{\vec{0} \in H} w(H) =0,$$
such that $\sum_H w(H)$ is minimized. Such an assignment $w$ of weights is called a {\it fractional almost $k$-cover} of $Q^n$. Denote by $f^*(n, k)$ the minimum of $\sum_H w(H)$, i.e. the minimum size of a fractional almost $k$-cover. We are able to  determine the precise value of $f^*(n, k)$ for every value of $n$ and $k$. 
\begin{theorem}\label{thm_frac}
For every $n$ and $k$, 
$$f^*(n, k)=\left(1+\frac{1}{2}+\cdots + \frac{1}{n}\right) k.$$
It implies that for fixed $n$ and $k \rightarrow \infty$,
$$f(n, k)= \left(1+\frac{1}{2}+\cdots + \frac{1}{n}+o(1)\right) k,$$
which grows linearly in $k$.
\end{theorem}

As an intermediate step of proving Theorem \ref{thm_frac}, we proved the following theorem, which can be viewed as an analogue of the well-known Lubell-Yamamoto-Meshalkin inequality \cite{bollobas, lubell, meshalkin, yamamoto} for subset sums. Moreover the inequality is tight for all non-zero binary vectors $\vec{a}=(a_1, \cdots, a_n)$.
\begin{theorem}\label{lym_analogue}
	Given $n$ real numbers $a_1, \cdots, a_n$, let 
	$$\mathcal{A}=\{S: \emptyset \neq S \subset[n], \sum_{i \in S} a_i = 1\}.$$
	Then $$\sum_{S \in \mathcal{A}}  \frac{1}{|S|\binom{n}{|S|}} \le 1.$$
	Equivalently, let $\mathcal{A}_t=\{S: S \in \mathcal{A}, |S|=t\}$, then 
	$$\sum_{t=1}^n \frac{|\mathcal{A}_t|}{t\binom{n}{t}} \le 1.$$
\end{theorem}

The rest of the paper is organized as follows. In the next section, we resolve the almost $3$-cover case, and show that the answer to the almost $4$-cover problem has only two possible values, thus proving Theorem \ref{thm_3}. Section \ref{sec_frac} contains the proofs of Theorems \ref{thm_frac} and \ref{lym_analogue}. The final section contains some concluding remarks and open problems.

\section{Almost $3$-covers of the $n$-cube}\label{sec_3-cover}
The following Punctured Combinatorial Nullstellensatz was proven by Ball and Serra (Theorem 4.1 in \cite{ball-serra}). Let $\mathbb{F}$ be a field and $f$ be a non-zero polynomial in  $\mathbb{F}[x_1, \cdots, x_n]$. We say $\vec{a}=(a_1, \cdots, a_n)$ is a {\it zero of  multiplicity $t$} of $f$, if $t$ is the minimum degree of the terms that occur in $f(x_1+a_1, \cdots, x_n+a_n)$. 
\begin{lemma} \label{lem_pcn}
 For $i=1, \cdots, n$, let $D_i \subset S_i \subset \mathbb{F}$ and $g_i=\prod_{s \in S_i} (x_i-s)$ and $\ell_i=\prod_{d \in D_i} (x_i-d)$.
	If $f$ has a zero of multiplicity at least $t$ at all the common zeros of $g_1, \cdots, g_n$, except at least one point of $D_1 \times \cdots \times D_n$ where it has a zero of multiplicity less than $t$, then there are polynomials $h_\tau$ satisfying $\deg(h_\tau) \le \deg(f)-\sum_{i \in \tau} \deg(g_i)$, and a non-zero polynomial $u$ satisfying $\deg(u) \le \deg(f)-\sum_{i=1}^n (\deg(g_i)-\deg(\ell_i))$, such that 
	$$f=\sum_{\tau \in T(n, t)} g_{\tau(1)}\cdots g_{\tau(t)}h_{\tau}+u \prod_{i=1}^n \frac{g_i}{\ell_i}.$$
	Here $T(n, t)$ indicates the set of all non-decreasing sequences of length $t$ on $[n]$.
\end{lemma}

This punctured Nullstellensatz will be our main tool in proving Theorem \ref{thm_3}. We start with the $k=3$ case. 
\begin{theorem}\label{thm_1}
For $n \ge 2$, $f(n, 3)=n+3$.
\end{theorem}
\begin{proof}
To show that $f(n, 3)=n+3$, it suffices to establish the lower bound. We prove by contradiction. Suppose $H_1, \cdots, H_{n+2}$ are $n+2$ affine hyperplanes that form an almost $3$-cover of $Q^n$. Without loss of generality, assume the equation defining $H_i$ is  $\langle \vec{b}_i, \vec{x} \rangle = 1$, for some non-zero vector $\vec{b}_i \in \mathbb{R}^n$. Define 
$P_i=\langle \vec{b}_i, \vec{x} \rangle - 1$, and let
$$f=P_1 P_2 \cdots P_{n+2}.$$
Since $H_1, \cdots, H_{n+2}$ form an almost $3$-cover of $Q^n$, every binary vector $\vec{x} \in Q^n \setminus \{\vec{0}\}$ is a zero of multiplicity at least $3$ of the polynomial $f$. We apply Lemma \ref{lem_pcn} with 
$$D_i=\{0\},~~S_i=\{0, 1\},~~g_i=x_i(x_i-1),~~\ell_i=x_i,$$
and write $f$ in the following form:
$$f=\sum_{1 \le i \le j \le k \le n} x_i(x_i-1)x_j(x_j-1)x_k(x_k-1) h_{ijk} + u \prod_{i=1}^n (x_i-1), $$
with $\deg(u) \le \deg(f)-n = 2$.

Note that $f=0$ on $Q^n \setminus \{\vec{0}\}$. Moreover,
$$\frac{\partial f}{\partial x_i}=\sum_{j=1}^{n+2} P_1 \cdots P_{j-1} \cdot \frac{\partial P_j}{\partial x_i} \cdot P_{j+1}\cdots P_{n+2}.$$ 
Recall that $P_j$ is a polynomial of degree $1$, thus $\partial f/\partial x_i$ is just a linear combination of $P_1 \cdots \hat{P_j} \cdots P_{n+2}$. Note that removing a single hyperplane still gives an almost $2$-cover. Therefore $\partial f/\partial x_i$ vanishes on $Q^n \setminus \{\vec{0}\}$. One can similarly show that all the second order partial derivatives of $f$ vanish on $Q^n \setminus \{\vec{0}\}$ as well. More generally, if $f$ is the product of equations of the affine hyperplanes from an almost $k$-cover, then all the $j$-th order derivatives of $f$ vanish on $Q^n \setminus \{\vec{0}\}$, for $j=0, \cdots, k-1$. It is not hard to observe that $x_i(x_i-1)x_j(x_j-1)x_k(x_k-1)h_{ijk}=g_ig_jg_k h_{ijk}$ also has its $t$-th order partial derivatives vanishing on the entire cube $Q^n$, for $t \in \{0, 1, 2\}$, since $x_i(x_i-1)=0$ on $Q^n$. Therefore the following polynomial
$$h=u \prod_{i=1}^n (x_i-1)$$
has $j$-th order partial derivatives vanishing on $Q^n \setminus \{\vec{0}\}$, for $j=0, 1, 2$.

We denote by $e_i$ the $n$-dimensional unit vector with the $i$-th coordinate being $1$. By calculations,
\begin{align*}
\frac{\partial h}{\partial x_i} = \frac{\partial u}{\partial x_i} \prod_{j=1}^n (x_j-1) + u \prod_{j \ne i} (x_j-1).
\end{align*}
Therefore 
$$0=\frac{\partial h}{\partial x_i}(e_i)=(-1)^{n-1} u(e_i),$$
and this implies 
$$u(e_i)=0~~\textrm{for~}i=1, \cdots, n.$$
Furthermore, 
\begin{align*}
\frac{\partial^2 h}{\partial x_i^2} = \frac{\partial^2 u}{\partial x_i^2} \prod_{j=1}^n (x_j-1) + 2 \frac{\partial u}{\partial x_i} \prod_{j \ne i} (x_j-1).
\end{align*}
Therefore 
$$0=\frac{\partial^2 h}{\partial x_i^2}(e_i)=(-1)^{n-1} \cdot 2 \frac{\partial u}{\partial x_i}(e_i),$$
and this implies 
$$\frac{\partial u}{\partial x_i}(e_i)=0~~\textrm{for~}i=1, \cdots, n.$$
Finally, 
\begin{align*}
\frac{\partial^2 h}{\partial x_i x_j} = \frac{\partial^2 u}{\partial x_i x_j} \prod_{k=1}^n (x_k-1) + \frac{\partial u}{\partial x_i} \prod_{k \ne j} (x_k-1) + \frac{\partial u}{\partial x_j} \prod_{k \ne i} (x_k-1) + u \prod_{k \ne i, j} (x_k-1)
\end{align*}
By evaluating it on $e_i$ and $e_i+e_j$, we have
$$\frac{\partial u}{\partial x_j}(e_i)=u(e_i)=0,~~\textrm{and}~~u(e_i+e_j)=0.$$
Summarizing the above results $u$ is a polynomial of degree at most $2$, satisfying: (i) $u=0$ at $e_i$ and $e_i+e_j$; (ii) $\partial u/\partial x_i=0$ at $e_j$ (possible to have $i=j$). We define a new single-variable polynomial $w$,
$$w(x)=u(x \cdot e_i+ e_j).$$
Then $\deg(w) \le 2$, and $w(0)=w(1)=w'(0)=0$, which implies $w \equiv 0$. Let 
$$u=\sum_{i} a_{ii}x_i^2 + \sum_{i<j} a_{ij}x_i x_j + \sum_i b_i x_i + c.$$ 
This gives for all $i \neq j$, $$a_{ii}=0,~~a_{ij}+b_i=0,~~a_{ii}+b_i+c=0.$$
On other other hand $\partial u/\partial x_i=0$ at $e_i$ gives
$$2a_{ii}+b_i=0.$$ 
It is not hard to derive from these equalities that
$$a_{ii}=a_{ij}=b_i=c=0.$$
Therefore $u \equiv 0$. But then we have 
$f(\vec{0})=0$, which contradicts the assumption that $\vec{0}$ is not covered by any of the $n+2$ affine hyperplanes. Therefore $f(n, 3)=n+3$ for $n \ge 2$. Note that $f(1, 3)=3$ and the proof does not work for $n=1$ because $e_i+e_j$ does not exist in a $1$-dimensional space.
\end{proof}

Note that Theorem \ref{thm_1} already implies $f(n, 4)\ge f(n, 3)+1=n+4$ for all $n \ge 2$. For $n=2$, it is straightforward to check that $f(2, 4)=6$, with an optimal almost $4$-cover $x_1=1$ (twice), $x_2=1$ (twice), and $x_1+x_2=1$ (twice). However for $n \ge 3$, we can improve this lower bound by $1$.

\begin{theorem} \label{thm_2}
For $n \ge 3$, $f(n, 4) \in \{n+5, n+6\}$. Moreover, for $3 \le n \le 5$, $f(n, 4)=n+5$.
\end{theorem}
\begin{proof}
Suppose $n \ge 3$, we would like to prove by contradiction that $n+4$ affine hyperplanes cannot form an almost $4$-cover of $Q^n$. Following the notations in the previous proof, we have 
$$P_1 \cdots P_{n+4}=f=\sum_{1 \le i \le j \le k \le l \le n} g_ig_jg_kg_l h_{ijkl} + u \prod_{i=1}^n (x_i-1),$$
with $\deg(u) \le 4$. Following similar calculations, $u$ satisfies the following relations: (i) $u=0$ at $e_i$, $e_i+e_j$ and $e_i+e_j+e_k$ for distinct $i, j, k$; (ii) $\partial u/\partial x_i=0$ at $e_j$  and $e_j+e_k$ for distinct $j, k$ ($i=j$ or $i=k$ possible); (iii) $\partial^2 u/\partial x_i^2=0$ at $e_j$ ($i=j$ possible); (iv) $\partial^2 u/\partial x_i \partial x_j=0$ at $e_k$ ($i=k$ or $j=k$ possible). Suppose
$$u=\sum a_{iiii}x_i^4 + \sum a_{iiij}x_i^3x_j  + \cdots + \sum b_{iii}x_i^3 + \cdots + \sum c_{ii}x_i^2 + \cdots+\sum d_i x_i + e.$$ 
Since $f(\vec{0})=(-1)^{n+4}=(-1)^n$, we know that $u(\vec{0})=1$ and thus $e=1$.

Let $w(x)=u(x \cdot e_i+ e_j)$. Then $w(0)=w(1)=w'(0)=w'(1)=w''(0)=0$. Since $w(x)$ has degree at most $4$, we immediately have $w \equiv 0$. This gives 
\begin{equation}
a_{iiii}=0, \vspace{-0.1cm}
\end{equation}
\begin{equation}
a_{iiij}+b_{iii}=0.
\end{equation}
\begin{equation}\label{eq1}
a_{iijj}+b_{iij}+c_{ii}=0.
\end{equation}
\begin{equation}\label{eq2}
a_{ijjj}+b_{ijj}+c_{ij}+d_i=0.
\end{equation}
\begin{equation}
a_{jjjj}+b_{jjj}+c_{jj}+d_j+1=0
\end{equation}
Using $u(e_i)=0$, $\partial u/\partial x_i(e_i)=0$ and $\partial^2 u/\partial x_i^2 (e_i)=0$, we have
$$a_{iiii}+b_{iii}+c_{ii}+d_i+1=0,$$
$$4a_{iiii}+3b_{iii}+2c_{ii}+d_i=0,$$
$$12a_{iiii}+6b_{iii}+2c_{ii}=0.$$

Using $a_{iiii}=0$, we can solve this system of linear equations and get $b_{iii}=-1$, $c_{ii}=3$, $d_i=-3$. This implies $a_{iiij}=1$. Plugged into the equations \eqref{eq1} and \eqref{eq2}, we have:
$$a_{iijj}+b_{iij}=-3,$$
$$b_{iij}+c_{ij}=2.$$ 
Now using $\partial^2 u/\partial x_i \partial x_j(e_i)=0$, we have $3a_{iiij}+2b_{iij}+c_{ij}=0$, which gives
$$2b_{iij}+c_{ij}=-3.$$
The three linear equations above give $b_{iij}=-5$, $c_{ij}=7$, $a_{iijj}=2$.

For $n \ge 3$, we can also utilize the relation $\partial^2 u/(\partial x_i \partial x_j)=0$ at $e_k$. This gives $a_{ijkk}+b_{ijk}+c_{ij}=0$, hence
$$a_{ijkk}+b_{ijk}=-7.$$
Also $\partial u/(\partial x_i)=0$ at $e_j+e_k$ simplifies to 
$$a_{ijkk}+a_{ijjk}+b_{ijk}+3=0.$$
Together they give $b_{ijk}=-11$ and $a_{ijkk}=4$. Finally, by calculations 
\begin{align*}
u(e_i+e_j+e_k)&=3a_{iiii}+6a_{iiij}+3a_{iijj}+3a_{iijk}+3b_{iii}+6b_{iij}+b_{ijk}+3c_{ii}+3c_{ij}+3d_i+e\\
&=2 \neq 0. 
\end{align*}
This gives a contradiction. Therefore for $n \ge 3$, there is no $u$ of degree at most $4$ satisfying the aforementioned relations. This shows for $n \ge 3$, $f(n, 4) \ge n+5$. The proof does not work for $n<3$ because $e_i+e_j+e_k$ does not exist in a 1-dimensional or 2-dimensional space. Since $f(n, 4) \le n+ \binom{4}{2}=n+6$, it can only be either $n+5$ or $n+6$, proving the first claim in Theorem \ref{thm_2}.

To show that $f(n, 4)=n+5$ for $3 \le n \le 5$, we only need to construct almost $4$-covers of $Q^n$ using $n+5$ affine hyperplanes. For $Q^3$, note that $x_1=1$, $x_2=1$, $x_3=1$, and $x_1+x_2+x_3=1$ form an almost $2$-cover. Doubling it gives an almost $4$-cover of $Q^3$ with $8$ affine hyperplanes. For $Q^4$, the following $9$ affine hyperplanes form an almost $4$-cover: $x_1=1$, $x_2=1$, $x_3=1$, $x_4=1$, $x_1+x_4=1$, $x_2+x_4=1$, $x_3+x_4=1$, $x_1+x_2+x_3=1$, $x_1+x_2+x_3+x_4=1$. For $Q^5$, one can take $x_i=1$ for $i=1, \cdots 5$, together with $x_{i}+x_{i+1}+x_{i+2}=1$ for $i=1, \cdots, 5$, where the addition is in $\mathbb{Z}_5$.
\end{proof}

Now we can combine these two results we just obtained to prove Theorem \ref{thm_3}.
\begin{proof}[Proof of Theorem \ref{thm_3}]
The $k=3$ case has been resolved by Theorem \ref{thm_1}. On the other hand we have 
$$f(n, k) \ge f(n, k-1)+1,$$
since removing an affine hyperplane from an almost $k$-cover gives an almost $(k-1)$-cover. Therefore for $k \ge 4$ and $n \ge 3$, 
$$f(n, k) \ge f(n, 4) + (k-4) \ge n+5+(k-4)=n+k+1.$$
The upper bound follows from the construction in the introduction.
\end{proof}

\section{Fractional almost $k$-covers of the $n$-cube}\label{sec_frac}
In this section, we determine $f^*(n, k)$ precisely and prove Theorem \ref{thm_frac}. We first establish an upper bound by an explicit construction of almost $k$-covers.

\begin{lemma}\label{lem_construction}
(i) For every $n, k$, 
$$f^*(n, k)\le \left(1+\frac{1}{2}+\cdots+\frac{1}{n}\right)k.$$
(ii) When $k$ is divisible by  $nx$, with $x= lcm(\binom{n-1}{0},\binom{n-1}{1},\cdots,\binom{n-1}{n-1})$, we have
$$f(n, k)\le \left(1+\frac{1}{2}+\cdots+\frac{1}{n}\right)k.$$
\end{lemma}
\begin{proof}
For (ii), it suffices to show that when $k=nx$, we can find an almost $k$-cover of $Q^n$, using  $k(1+1/2+\cdots+1/n)$ hyperplanes. We can then replicate this process to upper bound $f(n,k)$ where $k$ is any multiple of $nx$.

For $j=1, \cdots, n$, we will use every affine hyperplane of the form $x_{i_1}+x_{i_2}+\cdots+x_{i_j}=1$ a total of 	$\frac{nx}{j \binom{n}{j}}$ times. This number is actually an integer since it is equal to $\frac{x}{\binom{n-1}{j-1}}$, 
and by definition, $x$ is divisible by all $\binom{n-1}{j-1}$.

There are $\binom{n}{j}$ affine hyperplanes in this form, so the total number of being used is 
\begin{align*}
\sum_{j=1}^{n} \frac{nx}{j\binom{n}{j}} \cdot\binom{n}{j}
=\sum_{j=1}^{n} \frac{nx}{j} =\left(1+\frac{1}{2}+\cdots+\frac{1}{n} \right) k
\end{align*}

This is the number of hyperplanes claimed. If we could show that they form an almost $nx$-cover of $Q^n$, then we can scale the weights by a constant factor to obtain a fractional almost $k$-cover of $Q^n$ for every $k$ and (i) follows immediately.

Now we must show that these affine hyperplanes cover each point the appropriate number of times. It is apparent that $(0,\cdots,0)$ is never covered. Because of the symmetric nature of our construction, we just need to check how many times we have covered a vertex that has $t$ ones as coordinates. It gets covered by $t\binom{n-t}{j-1}$ distinct hyperplanes of the form $x_{i_1}+x_{i_2}+\cdots+x_{i_j}=1$, each of which appears $\frac{nx}{j\binom{n}{j}}$ times. Thus, the total number of times a point with $t$ ones is covered is given by:

\begin{align*}
\sum_{j=1}^{n} \frac{nx}{j\binom{n}{j}} \cdot t\binom{n-t}{j-1}&=nxt\sum_{j=1}^{n} \frac{\binom{n-t}{j-1}}{j \binom{n}{j}} =nxt\sum_{j=1}^n \frac{(n-t)!(n-j)!}{(n-t-j+1)!n!}\\
&=nxt \cdot \frac{(n-t)!}{n!}\cdot\sum_{j=1}^{n} \frac{(n-j)!}{(n-t-j+1)!}\\
&=\frac{nx}{(t-1)!\binom{n}{t}} \sum_{j=1}^{n} \frac{(n-j)!}{(n-t-j+1)!}
=\frac{nx}{\binom{n}{t}} \sum_{j=1}^{n} \binom{n-j}{t-1}\\
&=\frac{nx}{\binom{n}{t}}\binom{n}{t}
=nx=k
\end{align*} 
\end{proof}

To establish the lower bound in Theorem \ref{thm_frac}, first we assign weights to each vertex of $Q^n$ we wish to cover. A vertex with $t$ ones as coordinates is given weight $\frac{1}{t\binom{n}{t}}$. Then the sum of the weights of all the vertices is:
\begin{align*}
\sum_{t=1}^{n} \binom{n}{t} \cdot\frac{1}{t\binom{n}{t}}
=\sum_{t=1}^{n} \frac{1}{t}
\end{align*}
So if we cover each vertex $k$ times, the sum over all affine hyperplanes of the weights of the vertices they cover is $k(1+1/2+\cdots+1/n)$. Thus, if we can show that no hyperplane can cover a set of vertices whose weights sum to more than $1$, we will have proven the lower bound. Given an affine hyperplane $H$ not containing $\vec{0}$, denote by $\mathcal{A}_t$ the set of vertices with $t$ ones covered by $H$. We wish to prove Theorem \ref{lym_analogue}, i.e. 
\[
\sum_{t=1}^{n}\frac{|\mathcal{A}_t|}{t\binom{n}{t}}\leq{1}.
\] 

In general, vertices of $Q^n \setminus \{\vec{0}\}$ correspond to nonempty subsets of $[n]$. It is worth noting that if the equation of $H$ is $a_1x_1+\cdots +a_nx_n=1$, and all coefficients $a_i$ are strictly positive, the subsets corresponding to the vertices it covers will form an antichain. By the Lubell-Yamamoto-Meshalkin inequality,  
$$\sum_{t=1}^{n}\frac{|\mathcal{A}_t|}{t\binom{n}{t}} \le \sum_{t=1}^{n}\frac{|\mathcal{A}_t|}{\binom{n}{t}}\leq{1}.$$

However, some coefficients $a_i$ may be non-positive. In order to consider a more general hyperplane, we will associate each vertex it covers to some permutations of $[n]$. Consider the vertex $(c_1,c_2,\cdots,c_n) \in Q^n$ where the coordinates which are ones are $c_{i_1},\cdots,c_{i_t}$. We will associate this vertex to the permutations, $(d_1,d_2,\cdots,d_n)$ of $[n]$ which begin with $\{i_1,i_2,\cdots,i_t\}$ in some order and also have $\sum_{k=1}^{j} a_{d_k}<1$ for $1\leq{j}<t$.

\begin{lemma}\label{lem_disjoint}
	No permutation of $[n]$ is associated to more than one vertex on the same hyperplane.
\end{lemma}
\begin{proof}
	Suppose for the sake of contradiction that a permutation is associated to two vertices, $v$ and $w$, of the same hyperplanes. They may have either the same or a different number of ones as coordinates.
	
	Suppose that $v$ and $w$ both have $a$ ones as coordinates. The permutations associated to $v$ have the $a$ indices where $v$ has a $1$ as their first $a$ entries and the permutations associated to $v$ will have the $a$ indices where $w$ has a $1$ as their first $a$ entries. However, $v$ and $w$ do not have their ones in the exact same places so the set of the first $a$ entries is not the same for any pair of a permutation associated to $v$ and a permutation associated to $w$.
	
	We are left to consider the case where $v$ and $w$ do not have the same number of ones as coordinates. Without loss of generality, $v$ has $a$ ones as coordinates and $w$ has $b$ ones as coordinates where $a>b$. Suppose the permutation associated to both of them begins with $(d_1,d_2,\cdots,d_b)$. By the restrictions on permutations associated to $v$, we have that $\sum_{j=1}^{b} a_{d_j} <1$. However, the conditions on permutations associated to $w$ tell us that $(d_1,d_2,\cdots,d_b)$ are precisely the indices where $w$ has a $1$ coordinate. This implies $\sum_{j=1}^{b} a_{d_j}=1$, giving a contradiction.
	
\end{proof}

\begin{lemma}\label{lem_counting}
	The total number of permutations associated to a vertex with $t$ ones as coordinates is at least $(t-1)!(n-t)!$
\end{lemma}
\begin{proof}
	There are $(n-t)!$ ways to arrange the indices other than $\{i_1,i_2,\cdots,i_t\}$, so it suffices to show that there exist at least $(t-1)!$ ways to order $\{i_1,i_2,\cdots,i_t\}$ as $(d_1,d_2,\cdots,d_t)$ such that we have $\sum_{k=1}^{j} a_{d_k}<1$ for $1\leq{j}<t$. We notice that $(t-1)!$ is the number of ways to order $\{i_1,i_2,\cdots,i_t\}$ around a circle (up to rotations, but not reflections). Thus it suffices to show that for each circular ordering of $\{i_1,i_2,\cdots,i_t\}$, we can choose a starting place from which we may continue clockwise and label the elements as $(d_1,d_2,\cdots,d_t)$ in such a way that $\sum_{k=1}^{j} a_{d_k}<1$ for all $1\leq{j}<t$.
	
	Equivalently, the values of $a_{i_k}$, which happen to sum to $1$, have been listed around a circle for $1\leq{k}\leq{t}$. We wish to find some starting point from which all the partial sums of up to $t-1$ terms from that point are less than $1$. We can subtract $1/t$ from each to give the equivalent problem of $t$ numbers, which sum to $0$, written around a circle and needing to find a starting point from which all the partial sums of $1\leq{j}\leq{t-1}$ terms are less than $1-\frac{j}{t}$. It suffices to find a starting point for which the aforementioned partial sums are at most $0$.
	
	Consider all possible sums of any number of consecutive terms along the circle and choose the largest. We will label the terms in this sum as $e_1,e_2,\cdots,e_m$ and continue to order clockwise around the circle $e_{m+1},e_{m+2},\cdots,e_{t}$. Choose the starting point to be $e_{m+1}$. If any of the partial sums $e_{m+1}+e_{m+2}+\cdots+e_{m+j}$ exceeds $0$, for $m+j\leq{t}$, we could have simply chosen $e_1,e_2,\cdots,e_{m+j}$ to get a larger sum than $e_1+e_2+\cdots+e_m$. Similarly, if $e_{m+1}+e_{m+2}+\cdots+e_t+e_1+e_2+\cdots+e_j>0$ for some $1\leq{j}<m$, then we can note that $(e_1+e_2+\cdots+e_t)+(e_1+e_2+\cdots+e_j)$ exceeds $e_1+e_2+\cdots+e_m$, and since $e_1+e_2+\cdots+e_t=0$, we have that $e_1+e_2+\cdots+e_j>e_1+e_2+\cdots+e_m$, a contradiction. Thus, if we start at $e_{m+1}$ and move clockwise around the circle, the first $t-1$ partial sums will be at most $0$, as desired.
	
\end{proof}

Combining the previous results, we prove Theorem \ref{lym_analogue}, which can be viewed as an analogue of the LYM inequality for partial sums.

\begin{proof}[Proof of Theorem \ref{lym_analogue}]
By definition, sets in $\mathcal{A}$ correspond to vertices of $Q^n$ covered by the hyperplane $H$ with equation $a_1x_1+ \cdots +a_nx_n=1$. From Lemma \ref{lem_disjoint} and \ref{lem_counting}, these vertices define disjoint collections of permutations of length $n$. Moreover if $S \in \mathcal{A}$ has size $t$ then there are at least $(t-1)!(n-t)!$ permutations associated to it. Since in total there are at most $n!$ permutations, we get
$$\sum_{S \in \mathcal{A}} (|S|-1)!(n-|S|)!\leq{n!},$$
which implies 
$$\sum_{S \in \mathcal{A}}  \frac{1}{|S|\binom{n}{|S|}} \le 1$$
as desired.
\end{proof}

Now we are ready to prove our main theorem in this section.
\begin{proof}[Proof of Theorem \ref{thm_frac}]
As mentioned before, we assign weight $\frac{1}{t \binom{n}{t}}$ to a vertex of $Q^n \setminus \{\vec{0}\}$ with $t$ ones as coordinates. By Lemma \ref{lym_analogue}, every affine hyperplane covers a set of vertices whose weights sum to at most $1$. Therefore in an optimal fractional almost $k$-cover $\{w(H)\}$, 
$$f^*(n, k)=\sum_H w(H) \ge k \cdot \sum_{t=1}^n \frac{\binom{n}{t}}{t \binom{n}{t}} = \left(\sum_{i=1}^n \frac{1}{i}\right) k.$$ 
With the upper bound proved in Lemma \ref{lem_construction}, we have 
$$f^*(n, k)=\left(\sum_{i=1}^n \frac{1}{i}\right) k.$$

For integral almost $k$-covers, note that $f(n, k) \ge f^*(n, k)$. Using Lemma \ref{lem_construction} again,
$$f(n, k)=f^*(n, k)=\left(\sum_{i=1}^n \frac{1}{i}\right) k,$$
whenever $nx$ divides $k$. For fixed $n$ and $k \rightarrow \infty$, note that $f(n, k)$ is monotone in $k$, which immediately implies
$$f(n, k)= \left(1+\frac{1}{2}+\cdots + \frac{1}{n}+o(1)\right) k.$$
\end{proof}
For small values of $n$, we can actually determine the value of $f(n, k)$ for every $k$. It seems that for large $k$, $f(n, k)$ is not far from its lower bound $\lceil f^*(n, k)\rceil$. Trivially $f(1, k)=k$.
\begin{theorem}
The following statements are true:\\
(i) $f(2, k)=\lceil{\frac{3k}{2}} \rceil$ for $k \ge 1$.\\
(ii) $f(3, k)=\lceil{\frac{11k}{6}} \rceil$ for $k \ge 2$ and $f(3, 1)=3$.
\end{theorem}
\begin{proof}
(i) From previous discussions, there exists an almost $2$-cover of $Q^2$ using $3$ affine hyperplanes. Therefore $f(2, k+2) \le f(2, k)+3$, and it suffices to check $f(2, 1)=2$ and $f(2, 2)=3$ which are both obvious.\\
(ii) There exists an almost $6$-cover of $Q^3$ using $11$ affine hyperplanes. Therefore $f(3, k+6) \le f(3, k) + 11$. It suffices to check $f(3, k)\le \lceil{\frac{11k}{6}} \rceil$ for $k=2, \cdots, 5$ and $k=7$. From $f(n, 2)=n+1$, we have $f(3, 2)=4$. $f(3, 3)\le 6$ follows from Theorem \ref{thm_3}. $f(3, 4) \le 8$ since $f(3, 4) \le  2f(3, 2)$. $f(3, 5) \le 10$ by taking each of $x_i=1$ twice, $x_1+x_2+x_3=1$ three times, and $x_1+x_2+x_3=2$ once.  $f(3, 7) \le 13$ follows from taking each of  $x_1=1$, $x_2=1$, $x_3=1$, $x_1+x_2=1$, $x_1+x_3=1$ twice, and $x_2+x_3=1$, $x_2+x_3-x_1=1$, $x_1+x_2+x_3=1$ once.
\end{proof}
With the assistance of a computer program, we also checked that $f(4, k)=\lceil \frac{25k}{12}\rceil$ for $k \ge 2$. $f(5, k)=\lceil \frac{137}{60}k \rceil$ for $k \ge 15$ except when $k \equiv 7 \pmod{60}$ where $f(5, k)=\lceil \frac{137}{60}k \rceil+1$. The following question is natural.

\begin{question}
Does there exist an absolute constant $C>0$ which does not depend on $n$, such that for a fixed integer $n$, there exists $M_n$, so that whenever $k \ge M_n$, 
$$f(n, k) \le \left(1+\frac{1}{2}+\cdots + \frac{1}{n}\right) k + C?$$
\end{question}
If so, it would show that $f(n, k)$ and $f^*(n, k)$ differ by at most a constant when $k$ is large.
\section{Concluding Remarks}
In this paper, we determine the minimum size of a fractional almost $k$-cover of $Q^n$, and find the minimum size of an integral almost $k$-cover of $Q^n$, for $k \le 3$. Note that $f(n, 1)=n$ for $n \ge 1$, $f(n, 2)=n+1$ for $n \ge 1$, and $f(n, 3)=n+3$ for $n \ge 2$. All of them attain the upper bound $f(n, k) \le n + \binom{k}{2}$ whenever $n$ is not too small. For larger $k$ the following conjecture seems plausible.
\begin{conjecture}\label{conj_main}
For an arbitrary fixed integer $k \ge 1$ and sufficiently large $n$, 
$$f(n, k) = n+ \binom{k}{2}.$$
In other words, for large $n$, an almost $k$-cover of $Q^n$ contains at least $n+ \binom{k}{2}$ affine hyperplanes.
\end{conjecture}
In particular, for $k=4$, although $f(n, k) \le n+5$ for $n \le 5$, we suspect that for $n \ge 6$, $n+6$ affine hyperplanes are necessary for an almost $4$-cover of $Q^n$. If we restrict our attention to almost $k$-covers of $Q^n$ which use each of the affine hyperplanes $x_i=1$ for $i=1,\cdots,n$, we see that Conjecture \ref{conj_main}, if true, will imply the following weaker conjecture:
\begin{conjecture}\label{conj_weaker}
For fixed $k \ge 1$ and sufficiently large $n$, suppose $H_1, \cdots, H_m$ are affine hyperplanes in $\mathbb{R}^n$ not containing $\vec{0}$, and they cover all the vectors with $t$ ones as coordinates at least $k-t$ times, for $t=1, \cdots, k-1$. Then $m \ge \binom{k}{2}$.
\end{conjecture}
If this conjecture is true, then the $\binom{k}{2}$ bound is the best possible, since one can take $k-t$ copies of $x_1 + \cdots +x_n=t$ for $t=1,\cdots,k-1$. We note that using our weights from earlier, and the fact that a hyperplane cannot cover vertices whose weights sum to more than $1$, we require:

\begin{align*}
m\geq{\sum_{t=1}^{k-1} (k-t)\binom{n}{t}\frac{1}{t\binom{n}{t}}}
=1-k+\sum_{t=1}^{k-1} \frac{k}{t}=(1-o(1)) k \ln k.
\end{align*}

\noindent {\bf Remark added. }Alon communicated to us that Conjecture \ref{conj_weaker} is true. With his permission, we include his proof using Ramsey-type arguments below. Let $n$ be huge, and let $S$ be a collection of $m$ affine hyperplanes $H_1, \cdots, H_m$ satisfying the assumptions in Conjecture \ref{conj_weaker} and $N=[n]$. Color each subset of size $k-1$ by the index of the first  hyperplane that covers it ($m$ colors), by Ramsey there is a large subset $N_1$ of $N$ so that all $(k-1)$-subsets of it are covered by the same hyperplane. Without loss of generality, the equation of this hyperplane is $\sum_i w_i x_i=1$ and it follows that for all $j \in N_1$, all $w_j$ are equal and hence all are equal $1/(k-1)$. Therefore this hyperplane cannot cover any $k-t$ subset of $N_1$ for $t \geq 2$. Now throw away this hyperplane and repeat the argument for subsets of size $k-2$ of $N_1$. Coloring each such subset by the pair of smallest two indices of the hyperplanes that cover it (${m \choose 2}$ colors), we get a monochromatic subset $N_2$ of $N_1$ and observe that here too each of these two hyperplanes whose equation is $\sum_i w_i x_i=1$ has $w_j=1/(k-2)$ for all $j \in N_2$. So these cannot be useful for covering smaller subsets of $N_2$, throw them away and repeat this process. After dealing with all subsets including those of size $1$ we get the assertion of the conjecture. \qed\\

Alon and F\"uredi \cite{alon-furedi} proved the following result using induction on $n-m$: for $n \ge m \ge 1$, then $m$ hyperplanes that do not cover all vertices of $Q^n$ miss at least $2^{n-m}$ vertices. Let $g(n, m, k)$ be the minimum number of 
vertices covered less than $k$ times by $m$ affine hyperplanes not passing through $\vec{0}$. The Alon-F\"uredi theorem shows $g(n, m, 1)=2^{n-m}$ for $m=1,\cdots,n$. For $k=2$, it is straightforward to show that for $m=1,\cdots,n+1$, we have:
\begin{equation}
g(n, m, 2)= 2^{n-m+1}
\end{equation}
This is because $m-1$ hyperplanes leave at least $2^{n-m+1}$ vertices uncovered, and with one more hyperplane, these vertices cannot be covered twice. Similarly, for $k \ge 3$, we can obtain a trivial lower bound $g(n, m, k) \ge 2^{n-m+k-1}$. On the other hand, suppose $f(d, k)=t$ for $d \le n$, then take the affine hyperplanes $H_1, \cdots, H_t$ in an almost $k$-cover of $Q^d$. Observe that $H_i \times \mathbb{R}^{n-d}$ is an affine hyperplane in $Q^n$ not containing $\vec{0}$. It is easy to see that $\{H_i \times \mathbb{R}^{n-d}\}$ covers all the vertices of $Q^n$ but those of the form $\vec{0} \times \{0, 1\}^{n-d}$ at least $k$ times. Therefore $g(n, t, k) \le 2^{n-d}$. Theorem \ref{thm_3} shows $f(d, 3)=d+3$ for $d \ge 2$, therefore $g(n, d+3, 3) \le 2^{n-d}$ or $g(n, m, 3) \le 2^{n-m+3}$ for $m \ge 5$. We believe that this upper bound is tight. Note that the trivial lower bound is $g(n, m, 3) \ge 2^{n-m+2}$.
\begin{conjecture}
\begin{align*}
g(n, m, 3)= \begin{cases}
2^{n}, & m=1, 2;\\
2^{n-1}, &m=3;\\
2^{n-m+3}, &m=4, \cdots, n+3.\\
\end{cases}
\end{align*}
\end{conjecture}
One can further ask the following question for arbitrary $k$.
\begin{question}
Is it true that for all $n, m, k$, 
$$g(n, m, k) = 2^{n-d},$$
where $d$ is the maximum integer such that $f(d, k) \le m$?
\end{question}

\end{document}